\documentclass[final,leqno]{siamltex}
\usepackage{epsfig}
\usepackage{amsmath}
\usepackage{amssymb}
\usepackage{tikz}
\usepackage{graphicx}
\usepackage[notcite,notref]{showkeys}
\newtheorem{algorithm}{Weak Galerkin Algorithm}
\newcommand{\bq}{{\bf q}}
\newcommand{\bn}{{\bf n}}
\newcommand{\bx}{{\bf x}}
\newcommand{\bt}{{\bf t}}

\def\T{{\mathcal T}}
\def\E{{\mathcal E}}
\def\Q{{\mathbb Q}}

\def\l{{\langle}}
\def\r{{\rangle}}

\def\bn{{\bf n}}
\def\bq{{\bf q}}

\def\bbQ{\mathbb{Q}}
\newcommand{\pT}{{\partial T}}

\def\3bar{{|\hspace{-.02in}|\hspace{-.02in}|}}

\def\a#1{\begin{align*}#1\end{align*}}  
 
\def\p#1{\begin{pmatrix}#1\end{pmatrix}} 
  \numberwithin{equation}{section}
\numberwithin{table}{section} \numberwithin{figure}{section}

\title{A stabilizer free weak Galerkin element method with supercloseness of order two}

\author{Ahmed AL-Taweel\thanks{Department of Mathematics and Statistics,
		University of Arkansas at Little Rock, Little Rock, AR, 72204 (asaltaweel@ualr.edu)}
\and  Xiaoshen Wang\thanks{Department of
Mathematics, University of Arkansas at Little Rock, Little Rock, AR
72204 (xxwang@ualr.edu)}\and Xiu Ye\thanks{Department of
Mathematics, University of Arkansas at Little Rock, Little Rock, AR
72204 (xxye@ualr.edu). This research was supported in part by
National Science Foundation Grant DMS-1620016.}
\and
Shangyou Zhang\thanks{Department of
Mathematical Sciences, University of Delaware, Newark, DE 19716 (szhang@udel.edu).}
}

\begin{document}

\maketitle

\begin{abstract}
The weak Galerkin (WG) finite element method is an effective and flexible general numerical techniques for solving partial diﬀerential equations. A  simple weak Galerkin finite element method is introduced for second order elliptic problems. First we have proved that stabilizers are no longer needed for this WG element. Then we have proved  the supercloseness of order two for the WG finite element solution. The numerical results confirm the theory.
\end{abstract}

\begin{keywords}
weak Galerkin, finite element methods, weak gradient, second-order
elliptic problems, stabilizer free, supercloseness.
\end{keywords}

\begin{AMS}
Primary: 65N15, 65N30; Secondary: 35J50
\end{AMS}
\pagestyle{myheadings}

\section{Introduction}\label{Section:Introduction}

The weak Galerkin finite element method is an effective and flexible numerical techniques for solving partial diﬀerential equations.
It is a natural extension of the standard Galerkin finite element method where classical derivatives were substituted by weakly defined
derivatives on functions with discontinuity. The WG method was first introduced in \cite{wy,wymix} and then has been applied to solve various partial differential equations
\cite{hmy,Lin2014,lyzz,mwy-helm,mwy-biha,mwy-soe,mwyz-maxwell, mwyz-interface,mwy-brinkman,Shields,ww-div-curl,wy-stokes}.

The main idea of weak Galerkin finite
element methods is the use of weak functions and their corresponding
 weak derivatives in algorithm design. For the second order
elliptic equation, weak functions have the form of $v=\{v_0,v_b\}$
with $v=v_0$ inside of each element and $v=v_b$ on the boundary of
the element. Both $v_0$ and $v_b$ can be approximated by polynomials
in $P_\ell(T)$ and $P_s(e)$ respectively, where $T$ stands for an
element and $e$ the edge or face of $T$, $\ell$ and $s$ are
non-negative integers with possibly different values. Weak
derivatives are defined for weak functions in the sense of
distributions.  For example, one may approximate a weak gradient in the polynomial space
$[P_m(T)]^d$. Various combination of $(P_\ell(T),P_s(e),[P_m(T)]^d)$
leads to different  weak Galerkin methods tailored for
specific partial differential equations.

A stabilizing/penalty term is often used in finite element methods with discontinuous approximations to enforce connection of discontinuous functions across element boundaries.
Removing stabilizers from weak Galerkin finite element methods will simplify formulations
and reduce programming complexity significantly. Stabilizer free WG finite element methods have been studied in \cite{yz, aw, yzz}. The idea is to increase the connectivity of a weak function cross element boundary by raising the degree of polynomials for computing weak derivatives. In \cite{yz}, it has been proved that for a WG element $(P_k(T),P_k(e),[P_{j}(T)]^d)$, the condition of $j\ge k+n-1$  guarantees a stabilizer free WG method, where $n$ is the number of edges/faces of an element.

In this paper,  a WG finite element $(P_k(T),P_{k+1}(e),[P_{k+1}(T)]^2)$ is investigated for second order elliptic problems. This new WG finite element leads to a stabilizer free weak Galerkin formulation. In addition, we have proved order two supercloseness  for the WG finite element solution, i.e. the WG solution approaches to the $L^2$ projection of the true solution with the convergence rates two order higher than the optimal convergence rate in both an energy norm and the $L^2$ norm.  The numerical results show high accuracy of the WG method and confirm our theory.

For simplicity, we demonstrate the idea by using the second order elliptic problem that seeks an
unknown function $u$ satisfying
\begin{eqnarray}
-\nabla\cdot (a\nabla u)&=&f\quad \mbox{in}\;\Omega,\label{pde}\\
u&=&g\quad\mbox{on}\;\partial\Omega,\label{bc}
\end{eqnarray}
where $\Omega$ is a polytopal domain in $\mathbb{R}^2$, $\nabla u$ denotes the gradient of
the function $u$, and $a$ is a symmetric
$2\times 2$ matrix-valued function in $\Omega$. We shall assume that
there exist two positive numbers $\lambda_1,\lambda_2>0$ such that
\begin{equation}\label{ellipticity}
\lambda_1 \xi^t\xi\le \xi^ta\xi\le \lambda_2 \xi^t\xi,\qquad\forall
\xi\in\mathbb{R}^2.
\end{equation}
Here $\xi$ is understood as a column vector and $\xi^t$ is the transpose
of $\xi$.

The paper is organized as follows.  In Section 2, we shall describe
a new WG scheme. Section 3 will discussion the well posedness of
the WG scheme. The error analysis for the WG solutions in an energy norm and in the $L^2$ norms will be investigated in Section 4 and Section 5 respectively.
In Section 7, we shall present some numerical results that confirm the theory developed in earlier sections.
Finally, technical proof of Lemma 3.1 will be presented in Appendix.

\section{Weak Galerkin Finite Element Schemes}

Let ${\cal T}_h$ be a shape regular partition of the domain $\Omega$ consisting of
triangles. Denote by ${\cal E}_h$
the set of all edges in ${\cal T}_h$, and let ${\cal
E}_h^0={\cal E}_h\backslash\partial\Omega$ be the set of all
interior edges. For every element $T\in \T_h$, we
denote by $h_T$ its diameter and mesh size $h=\max_{T\in\T_h} h_T$
for ${\cal T}_h$.

First, we adopt the following notations,
\begin{eqnarray*}
(v,w)_{\T_h} &=&\sum_{T\in\T_h}(v,w)_T=\sum_{T\in\T_h}\int_T vw d\bx,\\
 \l v,w\r_{\partial\T_h}&=&\sum_{T\in\T_h} \l v,w\r_\pT=\sum_{T\in\T_h} \int_\pT vw ds.
\end{eqnarray*}

For a given integer $k \ge 1$, define a weak Galerkin finite
element space associated with $\T_h$ as follows
\begin{equation}\label{vhspace}
V_h=\{v=\{v_0,v_b\}:\; v_0|_T\in P_k(T),\ v_b|_e\in P_{k+1}(e),\ e\subset\pT,  T\in \T_h\}
\end{equation}
and its subspace $V_h^0$ is defined as
\begin{equation}\label{vh0space}
V^0_h=\{v: \ v\in V_h,\  v_b=0 \mbox{ on } \partial\Omega\}.
\end{equation}
We would like to emphasize that any function $v\in V_h$ has a single
value $v_b$ on each edge $e\in\E_h$.

For $v=\{v_0,v_b\}\in V_h+H^1(\Omega)$, a weak gradient $\nabla_wv$ is a piecewise vector valued polynomial such that on each $T\in\T_h$,  $\nabla_w v \in [P_{k+1}(T)]^2$ satisfies
\begin{equation}\label{d-d}
  (\nabla_w v, \bq)_T = -(v_0, \nabla\cdot \bq)_T+ \langle v_b, \bq\cdot\bn\rangle_{\partial T}\qquad
   \forall \bq\in [P_{k+1}(T)]^2.
\end{equation}
In the above equation, we let $v_0=v$ and $v_b=v$ if $v\in H^1(\Omega)$.

Let $Q_0$ and $Q_b$ be the two element-wise defined $L^2$ projections onto $P_k(T)$ and $P_{k+1}(e)$ with $e\subset\pT$ respectively for each $T\in\T_h$. Define $Q_hu=\{Q_0u,Q_bu\}\in V_h$. Let $\Q_h$ be the element-wise defined $L^2$ projection onto $[P_{k+1}(T)]^2$ on each element $T\in\T_h$.

\begin{algorithm}
A numerical approximation for (\ref{pde})-(\ref{bc}) can be
obtained by seeking $u_h=\{u_0,u_b\}\in V_h$
satisfying $u_b=Q_bg$ on $\partial\Omega$ and  the following equation:
\begin{equation}\label{wg}
(a\nabla_wu_h,\nabla_wv)_{\T_h}=(f,\; v_0) \quad\forall v=\{v_0,v_b\}\in V_h^0.
\end{equation}
\end{algorithm}

We will need the following lemma in the error analysis.

\begin{lemma}
Let $\phi\in H^1(\Omega)$, then on any $T\in\T_h$,
\begin{eqnarray}
\nabla_w (Q_h \phi)&=& \Q_h\nabla  \phi,\label{key}\\
\nabla_w\phi &=&\Q_h\nabla\phi.\label{key1}
\end{eqnarray}
\end{lemma}
\begin{proof}
Using (\ref{d-d}) and  integration by parts, we have that for
any $\bq\in [P_{k+1}(T)]^2$
\begin{eqnarray*}
(\nabla_w Q_h\phi,\bq)_T &=& -(Q_0\phi,\nabla\cdot\bq)_T
+\langle Q_b\phi,\bq\cdot\bn\rangle_{\pT}\\
&=& -(\phi,\nabla\cdot\bq)_T
+\langle \phi,\bq\cdot\bn\rangle_{\pT}\\
&=&(\Q_h\nabla\phi,\bq)_T,
\end{eqnarray*}
and
\begin{eqnarray*}
(\nabla_w \phi,\bq)_T &=& -(\phi,\nabla\cdot\bq)_T
+\langle \phi,\bq\cdot\bn\rangle_{\pT}\\
&=&(\nabla \phi,\bq)_T=(\Q_h\nabla\phi,\bq)_T,
\end{eqnarray*}
which imply the equation (\ref{key}) and the equation (\ref{key1}).
\end{proof}

For any function $\varphi\in H^1(T)$, the following trace
inequality holds true (see \cite{wymix} for details):
\begin{equation}\label{trace}
\|\varphi\|_{e}^2 \leq C \left( h_T^{-1} \|\varphi\|_T^2 + h_T
\|\nabla \varphi\|_{T}^2\right).
\end{equation}

\section{Well Posedness}

For any $v\in V_h+H^1(\Omega)$, define two semi-norms
\begin{eqnarray}
\3bar v\3bar^2&=&(\nabla_wv,\nabla_wv)_{\T_h},\label{3barnorm}\\
\3bar v\3bar^2_1&=&(a\nabla_wv,\nabla_wv)_{\T_h}.\label{3barnorm-1}
\end{eqnarray}
It follows from (\ref{ellipticity}) that there exist two positive constants $\alpha$ and $\beta$ such that
\begin{eqnarray}
\alpha\3bar v\3bar\le \3bar v\3bar_1\le \beta\3bar v\3bar.\label{norm-equ}
\end{eqnarray}

We introduce a discrete $H^1$ semi-norm as follows:
\begin{equation}\label{norm}
\|v\|_{1,h} = \left( \sum_{T\in\T_h}\left(\|\nabla
v_0\|_T^2+h_T^{-1} \|  v_0-v_b\|^2_\pT\right) \right)^{\frac12}.
\end{equation}
It is easy to show that $\|v\|_{1,h}$ defines a norm in $V_h^0$.

Next we will show that $\3bar  \cdot \3bar$ also defines a norm in $V_h^0$ by proving
the equivalence of $\3bar\cdot\3bar$ and $\|\cdot\|_{1,h}$ in $V_h$. First we need the following lemma.

\medskip

The proof of the following lemma is long and  can be found in Appendix.

\begin{lemma}\label{ML}
	For any  $v\in V_{h}$, we have
	\begin{eqnarray}
\sum_{T\in {\cal T}_{h}}h_T^{-1}\|v_0-v_b\|^2_{\partial T}\leq C\3bar v\3bar^2.\label{happy0}
	\end{eqnarray}
\end{lemma}

\begin{lemma} There exist two positive constants $C_1$ and $C_2$ such
that for any $v=\{v_0,v_b\}\in V_h$, we have
\begin{equation}\label{happy}
C_1 \|v\|_{1,h}\le \3bar v\3bar \leq C_2 \|v\|_{1,h}.
\end{equation}
\end{lemma}

\medskip

\begin{proof}
For any $v=\{v_0,v_b\}\in V_h$, it follows from the definition of
weak gradient (\ref{d-d}) and integration by parts that on each $T\in\T_h$
\begin{eqnarray}\label{n-1}
(\nabla_wv,\bq)_T=(\nabla v_0,\bq)_T+\l v_b-v_0,
\bq\cdot\bn\r_\pT\quad \forall \bq\in [P_{k+1}(T)]^2.
\end{eqnarray}
By letting $\bq=\nabla_w v|_T$ in (\ref{n-1}) we arrive at
\begin{eqnarray*}
(\nabla_wv,\nabla_w v)_T=(\nabla v_0,\nabla_w v)_T+\l v_b-v_0,
\nabla_w v\cdot\bn\r_\pT.
\end{eqnarray*}
Letting $\bq=\nabla v_0|_T$ in (\ref{n-1}) implies
\begin{eqnarray}
(\nabla_wv,\nabla v_0)_T=(\nabla v_0,\nabla v_0)_T+\l v_b-v_0,
\nabla v_0\cdot\bn\r_\pT.\label{n40}
\end{eqnarray}
From the trace inequality (\ref{trace}) and the inverse inequality
we have
\begin{eqnarray*}
\|\nabla_wv\|^2_T &\le& \|\nabla v_0\|_T \|\nabla_w v\|_T+ \|
v_0-v_b\|_\pT \|\nabla_w v\|_\pT\\
&\le& \|\nabla v_0\|_T \|\nabla_w v\|_T+ Ch_T^{-1/2}\|
v_0-v_b\|_\pT \|\nabla_w v\|_T,
\end{eqnarray*}
which implies
$$
\|\nabla_w v\|_T \le C \left(\|\nabla v_0\|_T +h_T^{-1/2}\|v_0-v_b\|_\pT\right),
$$
and consequently
$$\3bar v\3bar \leq C_2 \|v\|_{1,h}.$$

Next we will prove $C_1 \|v\|_{1,h}\le \3bar v\3bar $.
It follows from (\ref{n40}),  the trace inequality and the inverse inequality,
$$
\|\nabla v_0\|_T^2 \leq \|\nabla_w v\|_T \|\nabla v_0\|_T
+Ch_T^{-1/2}\| v_0-v_b\|_\pT \|\nabla v_0\|_T,
$$
which implies
$$
\sum_{T\in\T_h}\|\nabla v_0\|^2_T \leq C(\sum_{T\in\T_h}h_T^{-1}\| v_0-v_b\|^2_\pT+\sum_{T\in\T_h}\|\nabla_w v\|^2_T).
$$
Combining the above estimate and (\ref{happy0}),
 we prove  the lower bound of (\ref{happy}) and complete the proof of the lemma.
\end{proof}

\medskip

\begin{lemma}
The weak Galerkin finite element scheme (\ref{wg}) has a unique
solution.
\end{lemma}

\smallskip

\begin{proof}
If $u_h^{(1)}$ and $u_h^{(2)}$ are two solutions of (\ref{wg}), then
$\varepsilon_h=u_h^{(1)}-u_h^{(2)}\in V_h^0$ would satisfy the following equation
$$
(a\nabla_w \varepsilon_h,\nabla_w v)_{\T_h}=0\qquad\forall v\in V_h^0.
$$
 Then by letting $v=\varepsilon_h$ in the above
equation and (\ref{norm-equ}), we arrive at
$$
\3bar \varepsilon_h\3bar^2 \le (a\nabla_w \varepsilon_h,\nabla_w \varepsilon_h)=0.
$$
It follows from (\ref{happy}) that $\|\varepsilon_h\|_{1,h}=0$. Since $\|\cdot\|_{1,h}$ is a norm in $V_h^0$, one has $\varepsilon_h=0$.
 This completes the proof of the lemma.
\end{proof}

\section{Error Estimates in Energy Norm}
The goal of this section is to establish some error estimates for
the weak Galerkin finite element solution $u_h$ arising from (\ref{wg}).
For simplicity of analysis, we assume that the coefficient tensor
$a$ in (\ref{pde}) is a piecewise constant matrix with respect to
the finite element partition $\T_h$. The result can be extended to
variable tensors without any difficulty, provided that the tensor
$a$ is piecewise sufficiently smooth.

Let $e_h=Q_hu-u_h$. Next we derive an error equation that $e_h$ satisfies.
Define a bilinear form  $\ell(u,v)$ by
\begin{eqnarray*}
\ell(u,v)&=& \langle a(\nabla u-\Q_h\nabla u)\cdot\bn,\;v_0-v_b\rangle_{\partial\T_h}.
\end{eqnarray*}

\begin{lemma}
For any $v\in V_h^0$, the error $e_h$ satisfies the following  equation
\begin{eqnarray}
(a\nabla_we_h,\nabla_wv)_{\T_h}=\ell(u,v).\label{ee}
\end{eqnarray}
\end{lemma}

\begin{proof}
For $v=\{v_0,v_b\}\in V_h^0$, testing (\ref{pde}) by  $v_0$  gives
\begin{equation}\label{m1}
(a\nabla u,\nabla v_0)_{\T_h}- \langle
a\nabla u\cdot\bn,v_0-v_b\rangle_{\partial\T_h}=(f,v_0).
\end{equation}
To obtain the above estimate, we use the fact that
$\langle a\nabla u\cdot\bn, v_b\rangle_{\partial\T_h}=0$.

It follows from integration by parts, (\ref{d-d}) and (\ref{key})  that
\begin{eqnarray}
(a\nabla u,\nabla v_0)_{\T_h}&=&(a\Q_h\nabla  u,\nabla v_0)_{\T_h}\nonumber\\
&=&-(v_0,\nabla\cdot (a\Q_h\nabla u))_{\T_h}+\langle v_0, a\Q_h\nabla u\cdot\bn\rangle_{\partial\T_h}\nonumber\\
&=&(a\Q_h\nabla u, \nabla_w v)_{\T_h}+\langle v_0-v_b,a\Q_h\nabla u\cdot\bn\rangle_{\partial\T_h}\nonumber\\
&=&(a \nabla_w Q_hu, \nabla_w v)_{\T_h}+\langle v_0-v_b,a\Q_h\nabla u\cdot\bn\rangle_{\partial\T_h}.\label{j1}
\end{eqnarray}
Combining (\ref{m1}) and (\ref{j1}) gives
\begin{eqnarray}
(a\nabla_w Q_hu,\nabla_w v)_{\T_h}&=&(f,v_0)+\ell(u,v).\label{j2}
\end{eqnarray}
The error equation follows from subtracting (\ref{wg}) from (\ref{j2}),
\begin{eqnarray*}
(a\nabla_we_h,\nabla_wv)_{\T_h}=\ell(u,v)\quad \forall v\in V_h^0.
\end{eqnarray*}
This completes the proof of the lemma.
\end{proof}

\begin{theorem} Let $u_h\in V_h$ be the weak Galerkin finite element solution of (\ref{wg}). Assume the exact solution $u\in H^{k+3}(\Omega)$. Then,
there exists a constant $C$ such that
\begin{equation}\label{err1}
\3bar Q_hu-u_h\3bar \le Ch^{k+2}|u|_{k+3}.
\end{equation}
\end{theorem}
\begin{proof}
By letting $v=e_h$ in (\ref{ee}) and using (\ref{norm-equ}), we have
\begin{eqnarray}
\3bar e_h\3bar^2&\le&(a\nabla_we_h, \nabla_we_h)_{\T_h}=|\ell(u,e_h)|.\label{eee1}
\end{eqnarray}

Using the Cauchy-Schwarz inequality, the trace inequality (\ref{trace}), (\ref{ellipticity}) and (\ref{happy}), we have
\begin{eqnarray}
|\ell(u,e_h)|&=&\left|\sum_{T\in\T_h}\langle a(\nabla u-\Q_h\nabla
u)\cdot\bn, e_0-e_b\rangle_\pT\right|\nonumber\\
&\le & C \sum_{T\in\T_h}\|\nabla u-\Q_h\nabla u\|_{\pT}
\|e_0-e_b\|_\pT\nonumber\\
&\le & C \left(\sum_{T\in\T_h}h_T\|(\nabla u-\Q_h\nabla u)\|_{\pT}^2\right)^{\frac12}
\left(\sum_{T\in\T_h}h_T^{-1}\|e_0-e_b\|_\pT^2\right)^{\frac12}\nonumber\\
&\le & Ch^{k+2}|u|_{k+3}\3bar e_h\3bar.\label{mmm1}
\end{eqnarray}
It  follows from (\ref{eee1}) and (\ref{mmm1}) that
\[
\3bar e_h\3bar^2 \le Ch^{k+2}|u|_{k+3}\3bar e_h\3bar,
\]
which implies (\ref{err1}). This completes the proof.
\end{proof}

\section{Error Estimates in $L^2$ Norm}

The standard duality argument is used to obtain $L^2$ error estimate.
Recall $e_h=\{e_0,e_b\}=Q_hu-u_h=\{Q_0u-u_0, Q_bu-u_b\}$.
The considered dual problem seeks $\Phi\in H_0^1(\Omega)$ satisfying
\begin{eqnarray}
-\nabla\cdot a\nabla\Phi&=& e_0\quad
\mbox{in}\;\Omega.\label{dual}
\end{eqnarray}
Assume that the following $H^{2}$-regularity holds
\begin{equation}\label{reg}
\|\Phi\|_2\le C\|e_0\|.
\end{equation}

\begin{theorem} Let $u_h\in V_h$ be the weak Galerkin finite element solution of (\ref{wg}). Assume that the
exact solution $u\in H^{k+3}(\Omega)$ and (\ref{reg}) holds true.
 Then, there exists a constant $C$ such that
\begin{equation}\label{err2}
\|Q_0u-u_0\| \le Ch^{k+3}|u|_{k+3}.
\end{equation}
\end{theorem}

\begin{proof}
By testing (\ref{dual}) with $e_0$ we obtain
\begin{eqnarray}\nonumber
\|e_0\|^2&=&-(\nabla\cdot (a\nabla\Phi),e_0)\\
&=&(a\nabla \Phi,\ \nabla e_0)_{\T_h}-\l
a\nabla\Phi\cdot\bn,\ e_0- e_b\r_{\partial\T_h},\label{jw.08}
\end{eqnarray}
where we have used the fact that $e_b=0$ on $\partial\Omega$.
Setting $\phi=\Phi$ and $v=e_h$ in (\ref{j1}) yields
\begin{eqnarray}
(a\nabla\Phi,\;\nabla e_0)_{\T_h}&=&(a\nabla_w Q_h\Phi,\;\nabla_w e_h)_{\T_h}+\l (a\bbQ_h\nabla\Phi)\cdot\bn,\ e_0-e_b\r_{\partial\T_h}.\label{j1-new}
\end{eqnarray}
Substituting (\ref{j1-new}) into (\ref{jw.08}) gives
\begin{eqnarray}
\|e_0\|^2&=&(a\nabla_w e_h,\ \nabla_w Q_h\Phi)_{\T_h}+
\l a(\bbQ_h\nabla\Phi-\nabla\Phi)\cdot\bn,\ e_0-e_b\r_{\partial\T_h}\nonumber\\
&=&(a\nabla_w e_h,\ \nabla_w Q_h\Phi)_{\T_h}+\ell(\Phi,e_h)\nonumber\\
&=&\ell(u,Q_h\Phi)+\ell(\Phi,e_h).\label{m2}
\end{eqnarray}

Using the Cauchy-Schwarz  inequality and (\ref{trace}), we obtain
\begin{eqnarray}
|\ell(u,Q_h\Phi)|&=&\left|\sum_{T\in\T_h} \langle a(\nabla u-\bbQ_h\nabla
u)\cdot\bn,\; Q_0\Phi-Q_b\Phi\rangle_\pT \right|\nonumber\\
&\le&C\sum_{T\in\T_h} \|\nabla u-\bbQ_h\nabla u\|_{\pT}\| Q_0\Phi-Q_b\Phi\|_\pT \nonumber\\
&\le& C\left(\sum_{T\in\T_h}\|\nabla u-\bbQ_h\nabla u\|^2_\pT\right)^{1/2}\left(\sum_{T\in\T_h}\|Q_0\Phi-Q_b\Phi\|^2_\pT\right)^{1/2}\label{1st-term}
\end{eqnarray}
From the trace inequality (\ref{trace}) and the definition of $Q_b$, we have
\begin{eqnarray*}
\left(\sum_{T\in\T_h}\|Q_0\Phi-Q_b\Phi\|^2_\pT\right)^{1/2} &\le&\left(\sum_{T\in\T_h}\|Q_0\Phi-\Phi\|^2_\pT+\|\Phi-Q_b\Phi\|^2_\pT\right)^{1/2}\\
&\le& C \left(\sum_{T\in\T_h}\|Q_0\Phi-\Phi\|^2_\pT\right)^{1/2}\le Ch^{\frac32}\|\Phi\|_2
\end{eqnarray*}
and
$$
\left(\sum_{T\in\T_h}\|a(\nabla u-\bbQ_h\nabla
u)\|^2_\pT\right)^{1/2} \leq Ch^{k+\frac32}\|u\|_{k+3}.
$$
Combining  the above two estimates with (\ref{1st-term}) gives
\begin{eqnarray}\label{1st-term-complete}
|\ell(u,Q_h\Phi)| \leq C h^{k+3} |u|_{k+3}
\|\Phi\|_2.
\end{eqnarray}

Using the Cauchy-Schwarz inequality, the trace inequality (\ref{trace}), (\ref{ellipticity}), (\ref{happy}) and (\ref{err1}), we have
\begin{eqnarray}
|\ell(\Phi,e_h)|&=&\left|\sum_{T\in\T_h}\langle a(\nabla \Phi-\Q_h\nabla
\Phi)\cdot\bn, e_0-e_b\rangle_\pT\right|\nonumber\\
&\le & C \left(\sum_{T\in\T_h}h_T\|(\nabla \Phi-\Q_h\nabla \Phi)\|_{\pT}^2\right)^{\frac12}
\left(\sum_{T\in\T_h}h_T^{-1}\|e_0-e_b\|_\pT^2\right)^{\frac12}\nonumber\\
&\le & Ch\|\Phi\|_{2}\3bar e_h\3bar\nonumber\\
&\le& Ch^{k+3}|u|_{k+3}\|\Phi\|_2.\label{t30}
\end{eqnarray}

Substituting  (\ref{1st-term-complete}) and (\ref{t30})  into (\ref{m2})  yields
$$
\|e_0\|^2 \leq C h^{k+3}|u|_{k+3} \|\Phi\|_2,
$$
which, combined with the regularity assumption (\ref{reg}), gives the error estimate (\ref{err2}).
\end{proof}

\section{Numerical Experiments}\label{Section:numerical-experiments}

\subsection{Example 1}
Consider problem (\ref{pde}) with $\Omega=(0,1)^2$ and $a=\p{1&0\\0&1} $.
The source term $f$ and the boundary value $g$ are chosen so that the exact solution is
  (non-symmetric, nonzero boundary value)
\begin{equation*}
    u(x,y)=\sin(  x)\sin(\pi y).
\end{equation*}
In this example, we use uniform grids shown in Figure \ref{grid1}.
In Table \ref{t1}, we list the errors and the orders of convergence.
We can see that we do have two orders of superconvergence in both norms.

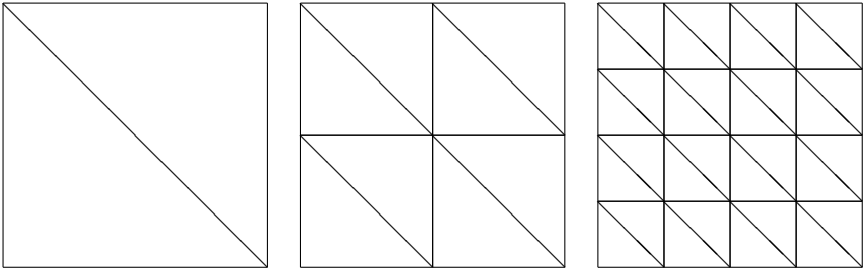
\begin{figure}[h!]
 \begin{center} \setlength\unitlength{1.25pt}
\begin{picture}(260,80)(0,0)
  \def\tr{\begin{picture}(20,20)(0,0)\put(0,0){\line(1,0){20}}\put(0,20){\line(1,0){20}}
          \put(0,0){\line(0,1){20}} \put(20,0){\line(0,1){20}}
   \put(0,20){\line(1,-1){20}}   \end{picture}}
 {\setlength\unitlength{5pt}
 \multiput(0,0)(20,0){1}{\multiput(0,0)(0,20){1}{\tr}}}

  {\setlength\unitlength{2.5pt}
 \multiput(45,0)(20,0){2}{\multiput(0,0)(0,20){2}{\tr}}}

  \multiput(180,0)(20,0){4}{\multiput(0,0)(0,20){4}{\tr}}

 \end{picture}\end{center}
\caption{The first three levels of triangular grids for Examples 1 and 2.}
\label{grid1}
\end{figure}

\begin{table}[h!]
  \centering \renewcommand{\arraystretch}{1.1}
  \caption{Example 1:  Error profiles and convergence rates on grids shown in Figure \ref{grid1}. }
\label{t1}
\begin{tabular}{c|cc|cc}
\hline
level & $\|u_h- Q_h u\|_0 $  &rate &  $\3bar u_h- Q_h u\3bar $ &rate    \\
\hline
 &\multicolumn{4}{c}{by the $P_1$-$P_2$ WG finite element } \\ \hline
 5&   0.5867E-06 & 3.99&   0.6420E-04 & 2.99  \\
 6&   0.3677E-07 & 4.00&   0.8043E-05 & 3.00 \\
 7&   0.2310E-08 & 3.99&   0.1021E-05 & 2.98 \\
\hline
 &\multicolumn{4}{c}{by the $P_2$-$P_3$ WG finite element } \\ \hline
 3&   0.6072E-05 & 4.90&   0.2699E-03 & 3.95 \\
 4&   0.1938E-06 & 4.97&   0.1704E-04 & 3.99\\
 5&   0.6113E-08 & 4.99&   0.1071E-05 & 3.99\\
\hline
 &\multicolumn{4}{c}{by the $P_3$-$P_4$ WG finite element } \\ \hline
 2&   0.1513E-04 & 5.50&   0.4704E-03 & 4.45\\
 3&   0.2397E-06 & 5.98&   0.1519E-04 & 4.95\\
 4&   0.3750E-08 & 6.00&   0.4819E-06 & 4.98\\
 \hline
\end{tabular}%
\end{table}%

\subsection{Example 2}
We solve problem (\ref{pde}) where $\Omega=(0,1)^2$ and
\a{ a=\p{2&1\\1&3}. }
The source term $f$ and the boundary value $g$ are chosen so that the exact solution is
\begin{equation*}
    u(x,y)=x^5y^2.
\end{equation*}
We use same meshes as Example 1.  The result is listed in Table \ref{t2}.
The superconvergence phenomena are same as those in Example 1, i.e., 2 orders of
   superconvergence in both $L^2$ norm and three-bar norm.

\begin{table}[h!]
  \centering \renewcommand{\arraystretch}{1.1}
  \caption{Example 2:  Error profiles and convergence rates on grids shown in Figure \ref{grid1}. }
\label{t2}
\begin{tabular}{c|cc|cc}
\hline
level & $\|u_h- Q_h u\|_0 $  &rate &  $\3bar u_h- Q_h u\3bar $ &rate    \\
\hline
 &\multicolumn{4}{c}{by the $P_1$-$P_2$ WG finite element } \\ \hline
 5&   0.5487E-06 & 3.95&   0.1023E-03 & 2.97\\
 6&   0.3471E-07 & 3.98&   0.1288E-04 & 2.99\\
 7&   0.2182E-08 & 3.99&   0.1621E-05 & 2.99\\
\hline
 &\multicolumn{4}{c}{by the $P_2$-$P_3$ WG finite element } \\ \hline
 3&   0.4230E-05 & 4.69&   0.3492E-03 & 3.84 \\
 4&   0.1440E-06 & 4.88&   0.2261E-04 & 3.95 \\
 5&   0.4680E-08 & 4.94&   0.1432E-05 & 3.98 \\
\hline
 &\multicolumn{4}{c}{by the $P_3$-$P_4$ WG finite element } \\ \hline
 2&   0.8490E-05 & 5.34&   0.5269E-03 & 4.62 \\
 3&   0.1436E-06 & 5.89&   0.1738E-04 & 4.92 \\
 4&   0.2296E-08 & 5.97&   0.5534E-06 & 4.97 \\
 \hline
\end{tabular}%
\end{table}%

\subsection{Example 3}
Consider problem (\ref{pde}) with $\Omega=(0,1)^2$ and $a=\p{1&0\\0&1} $.
The source term $f$ and the boundary value $g$ are chosen so that the exact solution is
\begin{equation*}
    u(x,y)=e^{\pi x} \sin(\pi y).
\end{equation*}
In this example, we use nonuniform grids shown in Figure \ref{grid2}.
In Table \ref{t3}, we list the errors and the orders of convergence.
We can see that we do have two orders of superconvergence in both norms.

\begin{figure}[htb]\begin{center}
\includegraphics[width=3.5in]{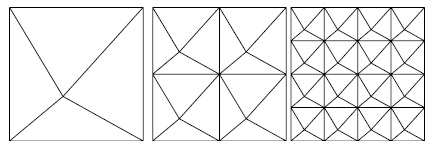} \

\caption{ The first three grids for the computation in Table \ref{t3}.  }
\label{grid2}
\end{center}
\end{figure}

\begin{table}[h!]
  \centering \renewcommand{\arraystretch}{1.1}
  \caption{Example 3:  Error profiles and convergence rates on grids shown in Figure \ref{grid2}. }
\label{t3}
\begin{tabular}{c|cc|cc}
\hline
level & $\|u_h- Q_h u\|_0 $  &rate &  $\3bar u_h- Q_h u\3bar $ &rate    \\
\hline
 &\multicolumn{4}{c}{by the $P_1$-$P_2$ WG finite element } \\ \hline
 5&   0.4530E-05 & 4.00&   0.1130E-02 & 3.00 \\
 6&   0.2830E-06 & 4.00&   0.1411E-03 & 3.00 \\
 7&   0.1768E-07 & 4.00&   0.1762E-04 & 3.00 \\
\hline
 &\multicolumn{4}{c}{by the $P_2$-$P_3$ WG finite element } \\ \hline
 5&   0.4027E-07 & 5.00&   0.1389E-04 & 4.00 \\
 6&   0.1256E-08 & 5.00&   0.8671E-06 & 4.00 \\
 7&   0.3920E-10 & 5.00&   0.5415E-07 & 4.00 \\
\hline
 &\multicolumn{4}{c}{by the $P_3$-$P_4$ WG finite element } \\ \hline
 4&   0.2138E-07 & 6.01&   0.4534E-05 & 5.01 \\
 5&   0.3308E-09 & 6.01&   0.1407E-06 & 5.01 \\
 6&   0.4979E-11 & 6.05&   0.4378E-08 & 5.01 \\
 \hline
\end{tabular}%
\end{table}%

\subsection{Example 4}
Consider problem (\ref{pde}) with $\Omega=(0,1)^3$ and $a=I_{3\times 3}$.
The source term $f$ and the boundary value $g$ are chosen so that the exact solution is
\begin{equation*}
    u(x,y,z)=\sin(\pi x)\sin(\pi y)\sin(\pi z).
\end{equation*}
We use tetrahedral meshes shown in Figure \ref{grid3}.
The results of the $P_k$-$P_{k+1}$ weak Galerkin finite element methods
    are listed in Table \ref{t4}.
The superconvergence phenomena are same as those in 2D, in first three
  examples,   two orders of superconvergence.

\begin{figure}[h!]
\begin{center}
 \setlength\unitlength{1pt}
    \begin{picture}(320,118)(0,3)
    \put(0,0){\begin{picture}(110,110)(0,0)
       \multiput(0,0)(80,0){2}{\line(0,1){80}}  \multiput(0,0)(0,80){2}{\line(1,0){80}}
       \multiput(0,80)(80,0){2}{\line(1,1){20}} \multiput(0,80)(20,20){2}{\line(1,0){80}}
       \multiput(80,0)(0,80){2}{\line(1,1){20}}  \multiput(80,0)(20,20){2}{\line(0,1){80}}
    \put(80,0){\line(-1,1){80}}\put(80,0){\line(1,5){20}}\put(80,80){\line(-3,1){60}}
      \end{picture}}
    \put(110,0){\begin{picture}(110,110)(0,0)
       \multiput(0,0)(40,0){3}{\line(0,1){80}}  \multiput(0,0)(0,40){3}{\line(1,0){80}}
       \multiput(0,80)(40,0){3}{\line(1,1){20}} \multiput(0,80)(10,10){3}{\line(1,0){80}}
       \multiput(80,0)(0,40){3}{\line(1,1){20}}  \multiput(80,0)(10,10){3}{\line(0,1){80}}
    \put(80,0){\line(-1,1){80}}\put(80,0){\line(1,5){20}}\put(80,80){\line(-3,1){60}}
       \multiput(40,0)(40,40){2}{\line(-1,1){40}}
        \multiput(80,40)(10,-30){2}{\line(1,5){10}}
        \multiput(40,80)(50,10){2}{\line(-3,1){30}}
      \end{picture}}
    \put(220,0){\begin{picture}(110,110)(0,0)
       \multiput(0,0)(20,0){5}{\line(0,1){80}}  \multiput(0,0)(0,20){5}{\line(1,0){80}}
       \multiput(0,80)(20,0){5}{\line(1,1){20}} \multiput(0,80)(5,5){5}{\line(1,0){80}}
       \multiput(80,0)(0,20){5}{\line(1,1){20}}  \multiput(80,0)(5,5){5}{\line(0,1){80}}
    \put(80,0){\line(-1,1){80}}\put(80,0){\line(1,5){20}}\put(80,80){\line(-3,1){60}}
       \multiput(40,0)(40,40){2}{\line(-1,1){40}}
        \multiput(80,40)(10,-30){2}{\line(1,5){10}}
        \multiput(40,80)(50,10){2}{\line(-3,1){30}}

       \multiput(20,0)(60,60){2}{\line(-1,1){20}}   \multiput(60,0)(20,20){2}{\line(-1,1){60}}
        \multiput(80,60)(15,-45){2}{\line(1,5){5}} \multiput(80,20)(5,-15){2}{\line(1,5){15}}
        \multiput(20,80)(75,15){2}{\line(-3,1){15}}\multiput(60,80)(25,5){2}{\line(-3,1){45}}
      \end{picture}}

    \end{picture}
    \end{center}
\caption{  The first three levels of grids used in Example 4. }
\label{grid3}
\end{figure}
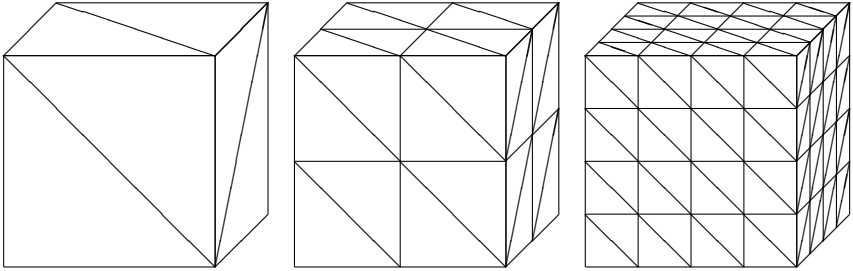

\begin{table}[h!]
  \centering \renewcommand{\arraystretch}{1.1}
  \caption{Example 4:  Error profiles and convergence rates on 3D grids shown in Figure \ref{grid3}. }
\label{t4}
\begin{tabular}{c|cc|cc}
\hline
level & $\|u_h- Q_h u\|_0 $  &rate &  $\3bar u_h- Q_h u\3bar $ &rate    \\
\hline
 &\multicolumn{4}{c}{by the $P_1$-$P_2$ WG finite element } \\ \hline
 1&     0.0748126&0.0&     0.9508644&0.0 \\
 2&     0.0120779&2.6&     0.1909763&2.3 \\
 3&     0.0010535&3.5&     0.0280478&2.8 \\
 4&     0.0000726&3.9&     0.0036750&2.9 \\
 5&     0.0000047&4.0&     0.0004657&3.0 \\
 6&     0.0000003&4.0&     0.0000584&3.0 \\
\hline
 &\multicolumn{4}{c}{by the $P_2$-$P_3$ WG finite element } \\ \hline
 1&     0.0106321&0.0&     0.4099718&0.0 \\
 2&     0.0019247&2.5&     0.0537392&2.9 \\
 3&     0.0000721&4.7&     0.0038352&3.8 \\
 4&     0.0000024&4.9&     0.0002476&4.0 \\
 5&     0.0000001&5.0&     0.0000156&4.0 \\
\hline
 &\multicolumn{4}{c}{by the $P_3$-$P_4$ WG finite element } \\ \hline
 1&     0.0170645&0.0&     0.3318098&0.0 \\
 2&     0.0003554&5.6&     0.0128696&4.7 \\
 3&     0.0000062&5.8&     0.0004504&4.8 \\
 4&     0.0000001&6.0&     0.0000145&5.0 \\ \hline
\end{tabular}%
\end{table}%

\appendix

\section*{Appendix}\label{Section:appendix}

We prove Lemma \ref{ML}  in the Appendix. First we need the following lemmas.

\smallskip

\begin{lemma}\label{eqn3}
	Let
	\begin{eqnarray*}
	Q_{1}(\bx)&=&\sum_{j=1}^{n}\sum_{i=1}^{n} a_{ij}x_{i}x_{j}=\bx^{T}A\bx,
	\end{eqnarray*}
and
	\begin{eqnarray*}
		Q_{2}(\bx)&=&\sum_{j=1}^{n}\sum_{i=1}^{n} b_{ij}x_{i}x_{j}=\bx^{T}B\bx,
\end{eqnarray*}
be two positive definite quadratic forms, where $\bx=(x_1,\cdots, x_n)$, $a_{ij}=a_{ji}$ and $b_{ij}=b_{ji}$. Then $\exists \lambda_{0}>0$ so that $\lambda Q_{1}(\bx)-Q_{2}(\bx)$ is a positive definite quadratic form for each $\lambda>\lambda_{0}$.
\end{lemma}

\begin{proof}
  We need to show that there exists a $\lambda_{0}>0$ such that $\lambda A-B$ is positive definite for $\lambda>\lambda_{0}$. Since $A$ is positive definite, $A=A^{\frac{1}{2}}A^{\frac{1}{2}}$, where $A^{\frac{1}{2}}$ is also symmetric positive definite with symmetric positive definite inverse. Thus
 \begin{eqnarray*}
 \lambda A-B=A^{\frac{1}{2}}\left( \lambda I-A^{\frac{-1}{2}}BA^{\frac{-1}{2}}\right) A^{\frac{1}{2}}.
 \end{eqnarray*}
Let $C=A^{\frac{-1}{2}}BA^{\frac{-1}{2}}$. Then $C$ is positive definite and thus
\begin{eqnarray}
C=P^{T}\begin{bmatrix}
\lambda_{0} &  &  &  \\
  &   & \ddots &   \\
 &  &  & \lambda_{n-1}
\end{bmatrix}P,
\end{eqnarray}
where $\lambda_{0}\geq\lambda_{1}\geq\dots\geq\lambda_{n-1}>0$ and $P$ is an orthogonal matrix. Thus we can write
\begin{eqnarray}
\lambda I-C=P^{T}\begin{bmatrix}
\lambda-\lambda_{0} &  &  &  \\
&   & \ddots & 0  \\
& 0 &  & \lambda-\lambda_{n-1}
\end{bmatrix}P.
\end{eqnarray}
Obviously, $\lambda I-C$ is positive definite when $\lambda>\lambda_{0}$ and so is $\lambda A-B$.
\end{proof}

\smallskip

\begin{lemma}\label{mainlemma}
	For any  $v\in V_{h}, \mbox{ if } \nabla_w v|_{T_{i}}\in [P_{k+1}(T_{i})]^{2}, \forall i=1,2, T_{1}\cap T_{2}=e_{1}$, then
	\begin{eqnarray}
		\|v^{(1)}_{0}-v^{(2)}_{0}\|^{2}_{e_{1}}\leq Ch_{T_{1}}\| \nabla_w \|^{2}_{T_{1}\cup T_{2}},\label{las1}
	\end{eqnarray}
where $v_{0}^{(i)}=v_{0}|_{T_{i}}, i=1,2$.
\end{lemma}
\begin{proof}
Without loss of generality, we may assume that the vertices of $T_{2}$ are $(0,0), (1,0), \mbox{ and } (0,1)$,
	\begin{eqnarray*}
		e_{1}=\{(x,0)|0\leq x\leq 1\},
	\end{eqnarray*}
	and the other edge of $T_{1}$ is $(a_{1},b_{1})$, where $b_{1}<0$.
	Denote $v_{0}^{(i)}=v_{0}|_{T_{i}}, i=1,2$.

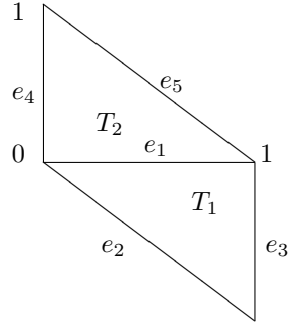
\begin{figure}[h!]
 \begin{center} \setlength\unitlength{2pt}
\begin{picture}(50,70)(0,-35)
 \put(0,0){\line(1,0){40}}\put(0,0){\line(0,1){30}}\put(0,30){\line(4,-3){40}}
 \put(0,0){\line(4,-3){40}}\put(40,0){\line(0,-1){30}}
 \put(-6,0){$0$} \put(41,0){$1$}  \put(-6,27){$1$}
 \put(-6,12){$e_4$} \put(22,14){$e_5$}  \put(19,2){$e_1$} \put(10,6){$T_2$}
 \put(11,-17){$e_2$} \put(28,-9){$T_1$}  \put(42,-17){$e_3$}
 \end{picture}\end{center}
		\caption{The mesh with two triangles}\label{fi1}
\end{figure}

	Let $\bt_{2}\mbox{ and } \bt_{4}$ be unit tangents to $e_{2} \mbox{ and } e_{4}$, respectively; $L_{3}\mbox{ and } L_{5}$ be linear functions such that $L_{3}|_{e_{3}}=0 \mbox{ and } L_{5}|_{e_{5}}=0$. Let
	\begin{eqnarray*}
		\bq_{1}=\bq|_{T_{1}}=L_{3}(x,y)\left( v_{b}^{(1)}-v_{0}^{(1)}+yQ^{(1)}_{k-1}\right) \bt_{2}=Q^{(1)}\bt_{2},
	\end{eqnarray*}
	and
	\begin{eqnarray*}
		\bq_{2}=\bq|_{T_{2}}=L_{5}(x,y)\left( v_{b}^{(2)}-v_{0}^{(2)}+yQ^{(2)}_{k-1}(x,y)\right) \bt_{4}=Q^{(2)}\bt_{4},
	\end{eqnarray*}
	where $Q_{k-1}^{(i)}$ are such that
	\begin{eqnarray}
	( Q^{(i)},p)_{T_{i}}=0,\quad\forall p\in P_{k-1}(T_{i}).\label{2.19 eq}
	\end{eqnarray}
	 By setting $v_{0}^{(2)}-v_{0}^{(1)}=0$ in $Q^{(i)}\mbox{ and } p=Q_{k-1}^{(i)}$ in (\ref{2.19 eq}), we know such $Q_{k+1}^{(i)}$ exist and are unique. Then
	\begin{eqnarray*}
		\bq_{i}\cdot\bn^{(2)}_{i}=0,\quad \bq_{i}|_{e_{2i+1}}=0, i=1,2.
	\end{eqnarray*}
	Scale $L_{1}\mbox{ and } L_{3}$ if necessary so that
	\begin{eqnarray*}
		-L_{3}(0,0)\bt_{2}\cdot\bn_{1}^{(1)}=1=L_{5}(0,0)\bt_{4}\cdot\bn_{1}^{(2)},
	\end{eqnarray*}
	where $\bn_{1}^{(i)}$ is the unit outwards normal vector of $e_{1}\in \partial T_{i},i=1,2$.\\ Since $L_{5}(1,0)\bt_{2}\cdot\bn_{1}^{(2)}=0$,
	\begin{eqnarray*}
		L_{5}(x,y)\bt_{4}\cdot\bn_{1}^{(2)}=\tilde{L}_{5}(x,y)=1-x-y.
	\end{eqnarray*}
	Similarly
	\begin{eqnarray*}
		L_{3}(x,y)\bt_{2}\cdot\bn_{1}^{(2)}=\hat{L}_{3}(x,y)=1-x+\alpha y, \mbox{ for some } \alpha.
	\end{eqnarray*}
	It follows from the shape regularity assumptions that the slope of $e_{3}, \frac{1}{\alpha}$, satisfies $|\frac{1}{\alpha}|\geq \alpha_{0}>0$ for some $\alpha_{0}$. Since $\tilde{L_{3}}|_{e_{1}}=\tilde{L_{5}}|_{e_{1}}$,
	\begin{eqnarray*}
		(\nabla_w v,\bq)_{T_{1}\cup T_{2}}&=&\left\langle v_{b}-v_{0},\bq\cdot\bn\right\rangle _{\partial
			T_{1}}+\left\langle v_{b}-v_{0},\bq\cdot\bn\right\rangle _{\partial T_{2}}\\&=&\left\langle v_{0}^{(2)}-v_{0}^{(1)}, (v_{0}^{(2)}-v_{0}^{(1)})\tilde{L_{3}}\right\rangle _{e_{1}}.
	\end{eqnarray*}
	Note that $0\leq \tilde{L_{5}}(x,y)\leq 1\mbox{ on } T_{1}\mbox{ and } 0\leq\tilde{L_{3}}(x,y)\leq 1\mbox{ on } T_{2}$. Write
	\begin{eqnarray*}
	(v_{0}^{(2)}-v_{0}^{(1)})|_{e_{1}}(x)=a_{0}+\dots+a_{k}x^{k}=P(x).
	\end{eqnarray*}
Let $\alpha=[a_{0},\dots,a_{k}]^{T}$. Then
\begin{eqnarray*}
\|v_{0}^{(2)}-v_{0}^{(1)}\|_{e_{1}}^{2}=\alpha^{T}A\alpha.
\end{eqnarray*}
	\begin{eqnarray*}
		(\nabla_w v, \bq)_{T_{1}\cup T_{2}}&=&\left\langle v_{0}^{(2)}-v_{0}^{(1)}, (v_{0}^{(2)}-v_{0}^{(1)})\tilde{L_{3}}(x,y)\right\rangle _{e_{1}}\\&=&\int_{0}^{1}(v_{0}^{(2)}-v_{0}^{(1)})^{2}(1-x)d x\\&=&\alpha^{T}B\alpha,
	\end{eqnarray*}
are positive definite quadratic forms in $a_{0},\dots,a_{k}$.
Note that
	\begin{eqnarray*}
		\|\bq\|_{T_{2}}^{2}&=&\int_{T_{2}}(1-x-y)^{2}(Q^{(2)})^{2}d A\\&\leq&C\left( \int_{T_{2}}(v_{b}^{(2)}-v_{0}^{(1)})dA+\int_{T_{2}}(Q_{k-1}^{(1)})^{2}dA\right)\\& \leq&C\left( \alpha^{T}D\alpha+\alpha^{T}E\alpha\right)
	\end{eqnarray*}
is a positive definite quadratic form. Similarly,
\begin{eqnarray*}
\|\bq\|_{T_{1}}^{2}=\alpha^{T}G\alpha
\end{eqnarray*}
is also a positive definite quadratic form. So by Lemma \ref{eqn3},
\begin{eqnarray*}
\|\bq\|_{T_{i}}^{2}\leq C\|v_{0}^{(2)}-v_{0}^{(1)}\|^{2}_{e_{1}}, i=1,2,
\end{eqnarray*}
and
\begin{eqnarray*}
C(\nabla_w v,\bq)_{T_{1}\cup T_{2}}\geq\|v_{0}^{(2)}-v_{0}^{(1)}\|_{e_{1}}^{2}.
\end{eqnarray*}
Thus
	\begin{eqnarray*}
		\|v_{0}^{(2)}-v_{0}^{(1)}\|_{e_{1}}^{2}&\leq& C|(\nabla_w v,\bq)_{T_{1}\cup T_{2}}|\\&\leq& C\left( \|\nabla_w v\|_{T_{1}}\|\bq\|_{T_{1}}+\|\nabla_w v\|_{T_{2}}\|\bq\|_{T_{2}}\right) \\&\leq& C\left( \|\nabla_w v\|_{T_{1}}+\|\nabla_w v\|_{T_{2}}\right)  \|v_{0}^{(2)}-v_{0}^{(1)}\|_{e_{1}},
	\end{eqnarray*}
for some $C$. It is worthwhile to note that this $C$ is independent of $h$ in some sense. If we want to change the scale from $1$ to $h$, we only need to replace $a_{i}$ by $a_{i}h^{i}$. Because of the dimensional differences of the two sides, $C$ will be replace by $Ch$. Thus after a scaling we have
	\begin{eqnarray*}
		 \|v_{0}^{(2)}-v_{0}^{(1)}\|_{e_{1}}\leq Ch_{T_{2}}^{\frac{1}{2}}\left( \|\nabla_w v\|_{T_{1}}+\|\nabla_w v\|_{T_{2}}\right) .
	\end{eqnarray*}
\end{proof}

\smallskip

\begin{lemma}\label{lastlemma}
Let $T_{1}$ and $T_{2}$ be such as in Lemma \ref{mainlemma},  then
	\begin{eqnarray}
\|v_{b}-v_{0}\|^{2}_{\partial T_{1}\cup \partial T_{2}}\leq Ch_{T_{1}}\|\nabla_w v\|^{2}_{T_{1}\cup T_{2}}.\label{20lemma}
	\end{eqnarray}
\end{lemma}

\begin{proof}
Without loss, we may assume that $T_{1}$ and $T_{2}$ are as shown in the Figure \ref{fi2}, where $e_{1}\cup e_{2}\cup e_{3}=\partial T_{1}, e_{1}\cup e_{4}\cup e_{5}=\partial T_{2}\mbox{ and } e_{1}=\partial T_{1}\cap \partial T_{2}$.

\begin{figure}[h!]
 \begin{center} \setlength\unitlength{2.5pt}
\begin{picture}(60,40)(-10,0)
 \put(0,0){\line(6,1){60}}\put(40,40){\line(2,-3){20}}\put(0,0){\line(0,1){40}}
 \put(0,40){\line(1,0){40}}\put(0,0){\line(1,1){40}}
  \put(-10,0){$(0,0)$} \put(-10,38){$(0,1)$} \put(42,38){$(1,1)$}  \put(58,6){$(p_1,p_2)$}
   \put(-6,18){$e_3$}  \put(25,0){$e_4$} \put(33,15){$T_2$} \put(15,29){$T_1$}
   \put(17,21){$e_1$}\put(17,41){$e_2$}  \put(55,22){$e_5$}
 \end{picture}\end{center}
		\caption{Two neighboring triangles}\label{fi2}
\end{figure}
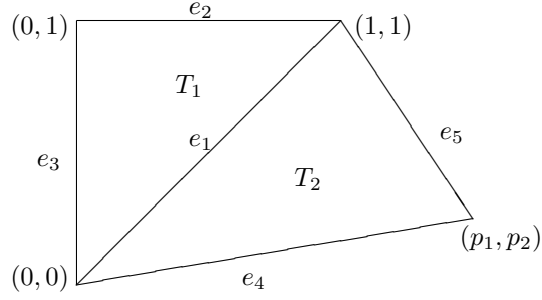

 Let us assume that $\nabla_w v|_{T_{1}\cup T_{2}}=0$. If follows from Lemma \ref{mainlemma} that
\begin{eqnarray*}
\|v_{0}^{(1)}-v_{0}^{(2)}\|^{2}_{e_{1}}\leq Ch_{T_{1}}\|\nabla_w v\|^{2}_{T_{1}\cup T_{2}}=0.
\end{eqnarray*}
We want to show
\begin{eqnarray}
	\|v_{b}-v_{0}\|^{2}_{e_{4}}\leq Ch_{T_{1}}\|\nabla_w v\|^{2}_{T_{1}\cup T_{2}}=0\label{eqnew1}
\end{eqnarray}
first.
 Let $L_{2}(x,y)=1-y$, then $L_{2}=0$ on $e_{2}$. Denote by $\bn_{j}^{(i)}$ the unite outer ward normal vector to $\partial T_{i} \cap e_{j}, i=1,2, j=1,\dots,5$. Let $\bt_{3}$ and $\bt_5$ be  unit tangent vectors to $e_{3}$ and $e_{5}$ respectively. Let
	\begin{eqnarray*}
	\bq|_{T_{1}}=\bq_{1}=Q^{(1)}\bt_{3},
	\end{eqnarray*}
where $Q^{(1)}=L_{2}Q_{k}^{(1)}, Q_{k}^{(1)}\in P_{k}(T_{1})$. Write
\begin{eqnarray*}
Q_{k}^{(1)}(x,y)=p_{k}(x)+(x-y)p_{k-1}(x,y),
\end{eqnarray*}
where $p_{k}(x)\in P_{k}(T_{1}), p_{k-1}(x,y)\in P_{k-1}(T_{1})$. We want to show that for each $p_{k}(x)$, we can find a unique $p_{k-1}(x,y)$ so that
\begin{eqnarray}
(Q^{(1)},p)=0,\quad \forall p\in P_{k-1}(T_{1}).
\end{eqnarray}
To do that we set $p_{k}=0$ and $p=p_{k-1}(x,y)$. Then
\begin{eqnarray*}
\int_{T_{1}}(1-y)(x-y)p_{k-1}^{2} dA=0,
\end{eqnarray*}
which implies that
\begin{eqnarray*}
p_{k-1}=0.
\end{eqnarray*}
This implies the existence and the uniqueness of $p_{k-1}$. Now let
\begin{eqnarray*}
\bq|_{T_{2}}=\bq_{2}=Q^{(2)}\bt_{5},
\end{eqnarray*}
where
\begin{eqnarray*}
Q^{(2)}(x,x)\bt_{5}\cdot\bn_{1}^{(2)}&=&-\bt_{3}\cdot\bn_{1}^{(1)}(1-x)p_{k}(x)\\&=&-\bt_{3}\cdot\bn_{1}^{(1)}Q^{(1)}(x,x).
\end{eqnarray*}
Thus
\begin{eqnarray*}
\left\langle \bq_{1}\cdot\bn_{1}^{(1)},v_{b}-v_{0}^{(1)}\right\rangle_{e_{1}}+\left\langle \bq_{2}\cdot\bn_{1}^{(2)},v_{b}-v_{0}^{(2)}\right\rangle_{e_{1}}=0,
\end{eqnarray*}
since $v_{0}^{(1)}=v_{0}^{(2)}$ and $\bq_{1}=\bq_{2} \mbox{ on } e_{1}$. Without loss, we assume $\bt_{5}\cdot\bn_{4}=1$.
Let $\delta$ be such that $\delta \bt_5\cdot\bn_1^{(2)}=-\bt_3\cdot\bn_1^{(1)}$.
 Write
\begin{eqnarray*}
	Q^{(2)}(x,y)=\delta (1-x)p_{k}(x)+(x-y)\left[ Q^{(2)}_{k}(x)+yR_{k-1}(x,y)\right],
\end{eqnarray*}
Write
\begin{eqnarray*}
(v_{b}-v_{0})|_{e_{4}}(x)&=&V_{k+1}(x)=v_{0}+\dots+v_{k+1}x^{k+1},\\
 p_{k}(x)&=&a_{0}+\dots+a_{k}x^k,
\end{eqnarray*}
and
\begin{eqnarray*}
xQ_{k}^{(2)}(x)=b_{1}x+\dots+b_{k+1}x^{k+1}
\end{eqnarray*}
such that
\begin{eqnarray*}
\delta a_{0}&=&v_{0},\\\delta (a_{1}+b_{1}-a_{0})&=&v_{1},\\\delta (a_{2}+b_{2}-a_{1})&=&v_{2}\\\vdots\\ \delta (b_{k+1}-a_{k})&=&v_{k+1}.
\end{eqnarray*}
Then
\begin{eqnarray*}
\delta(1-x)p_{k}(x)+xQ^{(2)}_{k}(x)=(v_{b}-v_{0})|_{e_{4}}
\end{eqnarray*}
and thus
\begin{eqnarray}
\left\langle Q^{(2)},v_{b}-v_{0}\right\rangle_{e_{4}}=\|v_{b}-v_{0}\|_{e_{4}}^{2}.\label{eq3n}\end{eqnarray}
Further, let
\begin{eqnarray*}
b_{1}=b_{2}=\dots=b_{k}=0.
\end{eqnarray*}
Then
\begin{eqnarray*}
	\delta a_{0}&=&v_{0},\\\delta a_{1}&=&v_{1}+v_{0},\\\vdots\\\delta a_{k}&=&v_{k}+v_{k-1},\\\delta b_{k+1}&=&v_{k+1}+v_{k}+v_{k-1}.
\end{eqnarray*}
Thus, $V=0$ implies  $p_{k}=Q_{k}^{(2)}=0$. Next we want to choose $R_{k-1}(x,y)$ so that
\begin{eqnarray}
(Q^{(2)},p)_{T_{2}}\equiv 0,\quad \forall p\in P_{k-1}(T_{2}).\label{eq4n}
\end{eqnarray}
To see if (\ref{eq4n}) has a unique solution, we set $V=0$. Then $p_{k}=Q_{k}^{(2)}=0$. Thus (\ref{eq4n}) becomes
\begin{eqnarray*}
\left( (x-y)yR_{k-1},p\right)_{T_{2}}=0,\quad \forall p\in P_{k-1}(T_{2}).
\end{eqnarray*}
It is easy to see that $R_{k-1}=0$. Thus
\begin{eqnarray*}
0=(\nabla_w v,\bq)_{T_{1}\cup T_{2}}=\left\langle v_{b}-v_{0},\bq\cdot\bn\right\rangle _{e_{4}}=\|v_{b}-v_{0}\|_{e_{4}}^{2}.
\end{eqnarray*}
Similarly, we can show that
\begin{eqnarray*}
\|v_{b}-v_{0}\|_{e_{i}}^{2}=0,\quad i=2,3,5.
\end{eqnarray*}
Now let's look at $\|v_{b}-v_{0}^{(1)}\|_{e_{1}}$. Since $\|\nabla_w v\|_{T_{1}}=0, \|v_{b}-v_{0}^{(1)}\|_{e_{i}}=0, i=2,3,$
\begin{eqnarray}
-(\nabla v,\bq)_{T_{1}}=\left\langle v_{b}-v_{0}^{(1)},\bq\cdot\bn_{1}\right\rangle_{e_{1}},\quad \forall\bq\in [P_{k+1}(T)]^{2}.\label{eq6n}
\end{eqnarray}
Let
\begin{eqnarray*}
\bq&=&(v_{b}-v_{0}^{(1)}+(x-y)Q_{k}(x,y))\bn_{1}\\&=&Q\bn_{1},
\end{eqnarray*}
where $Q_{k}$ is such that
\begin{eqnarray*}
(Q,p)_{T_{1}}=0, \quad \forall p\in P_{k-1}(T_{1})
\end{eqnarray*}
and $v_b$ is extended to $T_1$.
Then
\begin{eqnarray*}
0=\left\langle v_{b}-v_{0}^{(1)},\bq\cdot\bn\right\rangle_{e_{1}}= \|v_{b}-v_{0}^{(1)} \|_{e_{1}}.
\end{eqnarray*}
Similarly
\begin{eqnarray*}
	\|v_{b}-v_{0}^{(2)}\|_{e_{1}}=0.
\end{eqnarray*}
By (\ref{eq6n})
\begin{eqnarray*}
\nabla v_{0}^{(1)}=\nabla v_{0}^{(2)}=\vec{0}.
\end{eqnarray*}
Thus
\begin{eqnarray*}
\|\nabla_wv\|^{2}_{T_{1}\cup T_{2}}=0,
\end{eqnarray*}
implies
\begin{eqnarray*}
\|\nabla v_{0}\|_{T_{1}\cup T_{2}}^{2}+\|v_{b}-v_{0}\|^{2}_{\partial T_{1}\cup \partial T_{2}}=0.
\end{eqnarray*}
Let
\begin{eqnarray*}
\nabla v_{0}^{(i)}=\begin{bmatrix}
q_{1}^{(i)}\\q_{2}^{(i)}
\end{bmatrix},
\end{eqnarray*}
where $q_{j}^{(i)}\in P_{k-1}(T_{i})$. Let $\vec{P}_{k-1}=[1,x,y,\dots,y^{k-1}], \vec{P}_{k+1}=[1,x,y,\dots, y^{k+1}]$, write $q_{j}^{(i)}(x,y)=\vec{P}_{k-1}\vec{a_{j}}^{(i)}$, where $\vec{a_{j}}^{(i)}$ is the coefficient vector of $q_{j}^{(i)}$. Then
\begin{eqnarray*}
\|\nabla v_{0}^{(i)}\|_{T_{i}}^{2}=(\vec{a_{1}}^{(i)})^{T}A_{i}\vec{a_{1}}^{(i)}+(\vec{a_{2}}^{(i)})^{T}A_{i}\vec{a_{2}}^{(i)}.
\end{eqnarray*}
Thus $\|\nabla v_{0}\|_{T_{1}\cup T_{2}}^{2}$ is a positive definite quadratic form in $\alpha=\begin{bmatrix}
a_{1}^{(1)}\\a_{2}^{(1)}\\a_{1}^{(2)}\\a_{2}^{(2)}
\end{bmatrix}$. Similarly
\begin{eqnarray*}
\|v_{b}-v_{0}^{(1)}\|_{i}^{2}, i=1,2,3, \|v_{b}-v_{0}^{(2)}\|_{e_{j}}^{2},j=1,4,5
\end{eqnarray*}
are positive definite quadratic forms in $\beta_j,j=1,\cdots, 6$, where $\beta_{1}-\beta_{3}$ are variables in the first $3$ quadratic forms and $\beta_{4}-\beta_{6}$ are variables in the second $3$ quadratic forms, respectively. Note that the leading coefficients of $(v_{b}-v_{0}^{(1)})|_{e_{1}}$ and $(v_{b}-v_{0}^{(2)})|_{e_{1}}$ are the same. To keep those variables independent, we remove that coefficient from $\beta_{4}$ and let
 $\beta=\begin{bmatrix}
\beta_{1}\\\vdots\\\beta_{6}
\end{bmatrix}$. Write $\vec{\gamma}=\begin{bmatrix}
\alpha\\\beta
\end{bmatrix}$, then
\begin{eqnarray*}
\|\nabla v_{0}\|^{2}_{T_{1}\cup T_{2}}+\|v_{b}-v_{0}\|_{\partial T_{1}\cup \partial T_{2}}^{2}=\vec{\gamma}^{T}A\vec{\gamma},
\end{eqnarray*}
where $A$ is positive definite. Since $\nabla_w v|_{T_{1}\cup T_{2}}$ is uniquely determined by $\alpha$ and $\beta$ by solving a linear system,
\begin{eqnarray*}
\|\nabla_w v\|_{T_{1}\cup T_{2}}^{2}=\vec{\gamma}^{T}B\vec{\gamma},
\end{eqnarray*}
where $B$ is positive semi-definite. Since $\vec{\gamma}^{T}B\vec{\gamma}=0$ implies $\vec{\gamma}^{T}A\vec{\gamma}=0$, $B$ is positive definite. By Lemma \ref{eqn3} and a scaling argument,
\begin{eqnarray*}
\|v_{b}-v_{0}\|^{2}_{\partial T_{1}\cup \partial T_{2}}\leq Ch_{T_{1}}\|\nabla_w v\|_{T_{1}\cup T_{2}}^{2}.
\end{eqnarray*}
\end{proof}

\smallskip

Lemma \ref{ML} is a direct result of Lemma  \ref{lastlemma}.

\end{document}